\documentclass[reqno, 10pt, centertags]{amsart}
\usepackage{amsmath,amsthm,amscd,amssymb}
\usepackage{upref}
\usepackage{latexsym}
\usepackage{lscape}
\usepackage{graphicx}
\usepackage{tikz}
\usepackage{floatrow}

\makeatletter
\def\theequation{\@arabic\c@equation}

\newcommand{\e}{\hbox{\rm e}}

\newcommand{\bbR}{{\mathbb{R}}}
\newcommand{\R}{{\mathbb{R}}}

\newcommand{\C}{{\mathbb{C}}}

\newcommand{\cF}{{\mathcal F}}
\newcommand{\cG}{{\mathcal G}}
\newcommand{\cH}{{\mathcal H}}

\newcommand{\cL}{{\mathcal L}}

\newcommand{\cN}{{\mathcal N}}

\newcommand{\cQ}{{\mathcal Q}}

\newcommand{\cX}{{\mathcal X}}
\newcommand{\cY}{{\mathcal Y}}

\newcommand{\no}{\nonumber}
\newcommand{\lb}{\label}

\newcommand{\wti}{\widetilde  }

\newcommand{\bi}{\bibitem}

\newcommand{\mi}{\operatorname{Mas}}
\newcommand{\mo}{\operatorname{Mor}}
\newcommand{\bfi}{{\bf i}}
\newcommand{\sign}{{\rm sign}}

\numberwithin{equation}{section}

\newcommand{\dom}{\operatorname{dom}}

\newcommand{\tr}{\operatorname{tr}}
\newcommand{\re}{\operatorname{Re}}

\newcommand{\Sp}{\operatorname{Spec}}

\renewcommand{\Re}{\operatorname{Re }}
\renewcommand\Im{\operatorname{Im}}
\renewcommand{\ker}{\operatorname{ker}}

\theoremstyle{plain}
\newtheorem{theorem}{Theorem}[section]

\newtheorem{corollary}[theorem]{Corollary}
\newtheorem{proposition}[theorem]{Proposition}

\theoremstyle{definition}
\newtheorem{definition}[theorem]{Definition}

\begin{document}

\allowdisplaybreaks

\title[Maslov index]{Counting spectrum via the  Maslov index for one dimensional $\theta-$periodic Schr\"odinger operators}

\author[C.\ Jones]{Christopher K.\ R.\ T. Jones}
\address{Mathematics Department,
The University of North Carolina, Chapel Hill, NC 27599}
\email{ckrtj@email.unc.edu}
\author[Y. Latushkin]{Yuri Latushkin}
\address{Department of Mathematics,
The University of Missouri, Columbia, MO 65211, USA}
\thanks{Supported by the NSF grant  DMS-1067929, by the Research Board and Research Council of the University of Missouri, and by the Simons Foundation.
 We thank Graham Cox, Gregory Berkolaiko, and Alim Sukhtayev for discussions. Our special thanks go to Jared Bronski  for asking if the parameter $\theta$ can be used to compute the Maslov index, and to Fritz Gesztesy for help.}
\email{latushkiny@missouri.edu, sswfd@mail.missouri.edu}
\author[S. Sukhtaiev]{Selim Sukhtaiev	}
\date{\today}
\keywords{Schr\"odinger equation, Hamiltonian systems, eigenvalues, stability, differential operators, discrete spectrum}

\begin{abstract}
We study the spectrum of the Schr\"odinger operators with $n\times n$ matrix valued potentials on a finite interval subject to $\theta-$periodic boundary conditions. For two such operators, corresponding to different values of $\theta$, we compute the difference of their eigenvalue counting functions via the Maslov index of a path of Lagrangian planes. In addition we derive a formula for the derivatives of the eigenvalues with respect to $\theta$ in terms of the Maslov crossing form. Finally, we give a new shorter proof of a recent result relating the Morse and Maslov indices of the Schr\"odinger operator for a fixed $\theta$.
\end{abstract}
\maketitle

\section{introduction}
Relations between the spectral count and conjugate points for differential operators are of fundamental importance \cite{
Arn85,B56,CZ84,D76,
M63,S65}, and generalize the classical Sturm Theorem saying that the number of negative eigenvalues of the one dimensional scalar differential operator is equal to the number of conjugate points, that is, the zeros of the eigenfunction corresponding to the zero eigenvalue. In the matrix valued or multidimensional case the conjugate points are understood as the points of intersection of a path in the space of Lagrangian planes with the train of a fixed plane determined by the boundary conditions associated with the differential operator. The signed count of the conjugate points is called the Maslov index \cite{
Arn85, BbF95, CLM94, 
F04}.  Recently the relation between the spectral count and the Maslov index attracted much attention \cite{BM,DJ11,DP,CJLS,CJM1,CJM2,CDB06,CDB09,PW}. In particular, formulas relating the number of the negative eigenvalues of the Schr\"odinger operators with periodic potentials and the Maslov index were given in \cite{JLM,LSS}. In this paper we continue this latter work but give it a new spin using $\theta$ as a parameter generating the path of the Lagrangian planes.

The Schr\"odinger operator $H=-\partial_x^2+V$ with a periodic potential $V$ on $\R$ can be represented as the direct integral of $\theta-$periodic Schr\"odinger operators $H_{\theta}$ on $[0,2\pi]$, see \cite[Section XIII]{RS78}. This construction gives rise to the decomposition of the spectrum $\text{Spec}(H)=\cup_{k=1}^{\infty}[\alpha_k,\beta_k]$ into the union of spectral intervals $[\alpha_k,\beta_k]$, with the end points $\alpha_k,\beta_k$ being equal to the $k-$th eigenvalue of the operator with either periodic or anti-periodic boundary conditions. Thus the properties of each fiber operator $H_{\theta}$ and its eigenvalues are of great interest. A natural question in this context is how certain quantitative characteristics of $H_{\theta}$, such as the Morse index $\mo(H_{\theta})=N(0,\theta)$ or, more generally, the eigenvalue counting function $N(r,\theta)$, vary with respect to the parameter $\theta$. We address these questions by employing methods of symplectic geometry. 

Let $[a,b]$ be a finite interval and $V\in L^{\infty}([a,b], \R^{n\times n})$ be symmetric a.e., and consider the eigenvalue  problem for the operator $H_{\theta}:=(-\partial_x^2)_{\theta}+V$ defined in \eqref{1.1}-\eqref{1.3}, that is, the boundary value problem
\begin{align}
&-u''(x)+V(x)u(x)=\lambda u(x),\ \lambda\in\R,\ x\in[a,b], u=(u_1,\dots, u_n)\in \C^n,\lb{1.6}\\
&\quad u(b)=e^{\bfi \theta}u(a),\ u'(b)=e^{\bfi \theta}u'(a),\ \theta\in [0,2\pi) \lb{1.7}.
\end{align}
Problem \eqref{1.6},\eqref{1.7} can be reformulated in terms of existence of {\it conjugate points}. A pair of numbers $(\lambda,\theta)\in\R\times [0,2\pi)$ is called a conjugate point if the subspace $F_{\lambda}^2\subset\C^{4n}$ consisting of the Dirichlet and Neumann boundary traces of solutions to the second order equation \eqref{1.6} has a nontrivial intersection with the set $F_{\theta}^1:=\{(P,e^{\bfi \theta}P,-Q,e^{\bfi \theta}Q)^\top: P,Q\in \C^n\}$. By varying $\theta$ between the given values $\theta_1<\theta_2$ in $[0,\pi]$, and the spectral parameter $\lambda$ between $\lambda_{\infty}$, the uniform lower bound for the spectrum of $H_{\theta}$, and some fixed number $r>\lambda_{\infty}$, we obtain a loop $\gamma$ in the set of Lagrangian planes in $\R^{16n}$. Each intersection of this loop with the diagonal plane $\{(p,p):p\in \R^{8n}\}$ gives a conjugate point. The total number of the conjugate points, corresponding to the part of the loop where $\theta=\theta_1$(respectively, $\theta=\theta_2$), is equal to the number of the eigenvalues of $H_{\theta_1}$(respectively, $H_{\theta_2}$) located below $r$. Using the homotopy invariance property of the Maslov index, one concludes that the total number of conjugate points (counting their signs) is equal to zero. Therefore the difference between the eigenvalue counting functions for $H_{\theta_1}$ and $H_{\theta_2}$ can be evaluated via the Maslov index $\mi(\gamma|_{\lambda=r})$, the total number of the conjugate points (counting their signs) for the part of the loop where $\lambda=r$ is fixed, cf. Figure 1(I). Denoting by $N(r,\theta)$ the number of the eigenvalues of $H_{\theta}$ located below a fixed $r\in\R$, our main result therefore states that, see Theorem \ref{t2.2},
\begin{equation}\lb{i5}
		N(r,\theta_2)-N(r,\theta_1)=\mi(\gamma|_{\lambda=r}).
\end{equation}
Also, we prove its Corollary \ref{cor3.4} saying that
\begin{equation}\lb{i6}
\text{\ if $(\theta_1,\theta_2)\subset (0,\pi)\cup (\pi,2\pi)$, then\ }|N(r,\theta_2)-N(r,\theta_1)|\leq n.\
\end{equation}
In addition, we derive a formula for the derivative of the eigenvalues $\lambda(\theta)$ of $H_{\theta}$ with respect to $\theta$ in terms of the Maslov crossing form, see Theorem \ref{FE}.
\begin{figure}
	\begin{floatrow}
		
		\begin{picture}(100,100)(-20,0)
		\put(78,0){$r$}
		\put(6,8){\vector(0,1){95}}
		\put(5,10){\vector(1,0){95}}
		\put(71,40){\text{\tiny $\Gamma_2$}}
		\put(12,40){\text{\tiny $\Gamma_4$}}
		\put(45,73){\text{\tiny $\Gamma_3$}}
		\put(43,14){\text{\tiny $\Gamma_1$}}
		\put(43,-5){\text{I}}
		\put(100,12){$\lambda$}
		\put(10,98){$\theta$}
		\put(80,20){\line(0,1){60}}
		\put(10,20){\line(0,1){60}}
		\put(80,8){\line(0,1){4}}
		\put(10,2){$\lambda_\infty$}
		\put(10,8){\line(0,1){4}}
		\put(10,20){\line(1,0){70}}
		\put(10,80){\line(1,0){70}}
		\put(80,70){\circle*{4}}
		\put(20,20){\circle*{4}}
		\put(35,80){\circle*{4}}
		\put(45,80){\circle*{4}}
		\put(60,20){\circle*{4}}
		\put(70,20){\circle*{4}}
		\put(20,87){{\tiny \text{eigenvalues}}}
		\put(14,24){{\tiny \,\,\,\text{eigenvalues}}}  
		\put(85,23){\rotatebox{90}{{\tiny conjugate points}}}
		\put(82,16){{\tiny $\theta_1$}}
		\put(82,82){{\tiny $\theta_2$}}
		\put(0,25){\rotatebox{90}{{\tiny no intersections}}}
		\end{picture}
\ \ \ \ \ \ \ \ 
		\begin{picture}(100,100)(-20,0)
		\put(66,0){0}
		\put(76,0){$r$}
		\put(70,8){\vector(0,1){95}}
		\put(5,10){\vector(1,0){95}}
		\put(71,40){\text{\tiny $\Gamma_2$}}
		\put(12,40){\text{\tiny $\Gamma_4$}}
		\put(45,75){\text{\tiny $\Gamma_3$}}
		\put(43,14){\text{\tiny $\Gamma_1$}}
			\put(43,-5){\text{II}}
		\put(100,12){$\lambda$}
		\put(75,100){$s$}
		\put(80,20){\line(0,1){60}}
		\put(10,20){\line(0,1){60}}
		\put(80,8){\line(0,1){4}}
		\put(2,2){$\lambda^{\infty}$}
		\put(10,20){\line(1,0){70}}
		\put(10,80){\line(1,0){70}}
		\put(65,20){\circle*{4}}
		\put(80,50){\circle*{4}}
		\put(80,70){\circle*{4}}
		\put(20,80){\circle*{4}}
		\put(40,80){\circle*{4}}
		\put(60,80){\circle*{4}}
		\put(20,87){{\tiny \text{eigenvalues}}}
		\put(14,24){{\tiny \,\,\,\text{eigenvalues}}}  
		\put(85,23){\rotatebox{90}{{\tiny conjugate points}}}
		\put(82,18){{\tiny $\tau$}}
		\put(82,80){{\tiny $1$}}
		\put(0,25){\rotatebox{90}{{\tiny no intersections}}}
		\end{picture}
			\caption{\ }
	\end{floatrow}
\end{figure} 

A similar in spirit technique can be used to derive formulas for the difference between the Morse indices of $H_\theta$ and its rescaled version $H_{\theta}(t):=-t^{-2}(d^2/dx^2)_{\theta}+V(tx), t\in (0,1]$, here the parameter $\theta$ is fixed, and the role of the varying parameter is played by $t$. The latter setting has been already considered in \cite{JLM}. Section \ref{SectionT} of the current paper provides an alternative proof without using an augmented system employed in \cite{JLM}. 

We point out that the method briefly discussed above is quite general and well suited for Schr\"odinger operators  with self-adjoint boundary conditions on bounded domains in $\R^n$, e.g. Dirichlet, Neumann, Robin, $\vec{\theta}-$periodic. It has been successfully employed in \cite{CJLS}, \cite{CJM1}, \cite{CJM2}, \cite{DJ11}, \cite{LSS}. The novelty of the current approach is in variation of the parameter describing the boundary conditions, not the geometry of domain.      

{\bf Notations.} We denote by $I_n$ and $0_n$ the $n\times n$ identity and zero matrix.  For an $n\times m$ matrix $A=(a_{ij})_{i=1,j=1}^{n,m}$
and a $k\times\ell$ matrix $B=(b_{ij})_{i=1,j=1}^{k,\ell}$, we denote by
$A\otimes B$ the Kronecker product, that is, the $nk\times m\ell$ matrix composed of $k\times\ell$ blocks $a_{ij}B$, $i=1,\dots n$, $j=1,\dots m$. We let $(\cdot\,,\cdot)_{\R^n}$ denote the real scalar product in the space $\bbR^n$ of $n\times 1$ vectors, and let $\top$ denote transposition. When $a=(a_i)_{i=1}^n\in\bbR^n$ and $b=(b_j)_{j=1}^m\in\bbR^m$ are $(n\times 1)$ and $(m\times 1)$ column vectors,  we use notation $(a,b)^\top$ for the $(n+m)\times 1$ column vector with the entries $a_1,\dots,a_n,b_1,\dots,b_m$ (just avoiding  the use of $(a^\top,b^\top)^\top$). We denote by $\cL(\cX)$ the set of linear bounded operators and by $\Sp(T)=\Sp(T; \cX)$ the spectrum of an operator on a Hilbert space $\cX$. Finally, we use notation $J$ and $M_{\theta}$ for the following $2\times 2$ matrices
\begin{equation}\label{dfnJM}J:=\left[\begin{matrix}
0& 1\\
-1& 0\\
\end{matrix}\right],\ \ 
 M_{\theta}:=\left[\begin{matrix}
 \cos\theta & -\sin\theta\\
 \sin\theta & \cos\theta
 \end{matrix}\right],\ \theta\in[0,2\pi).
 \end{equation}
 

\section{Symplectic Approach To The Eigenvalue Problem}
In this section we introduce a framework for the sequel: firstly, we give a formal definition of the unperturbed $\theta-$periodic Laplacian on $L^2([a,b],\C^n),\ -\infty<a<b<\infty$, secondly, we discuss the symplectic approach to the eigenvalue problem, and, finally, we recall the definition of the Maslov index.

Given a $\theta\in[0,2\pi)$, we consider the operator $(-\partial_x^2)_{\theta}$, defined as follows
\begin{align}
(-\partial_x^2)_{\theta}&: L^2([a,b],\C^n) \rightarrow L^2([a,b],\C^n),\lb{1.1}\\
\dom(-\partial_x^2)_{\theta}&:=\Big\{u\in L^2([a,b],\C^n): u, u'\in AC([a,b],\C^n), u''\in L^2([a,b],\C^n)\no\\
&\qquad\qquad u(b)=e^{\bfi\theta}u(a)\ \text{and\ } u'(b)=e^{\bfi\theta}u'(a)\Big\},\lb{1.2}\\
(-\partial_x^2)_{\theta}u&:=-u'',\ \ u\in \dom(-\partial_x^2)_{\theta}.\lb{1.3}
\end{align}  
The operator $(-\partial_x^2)_{\theta}$ is self-adjoint, non-negative, and its spectrum is discrete (see, e.g., \cite{RS78} for more details). Next, we consider the Schr\"odinger operator $H_{\theta}:= (-\partial_x^2)_{\theta}+V$ with a matrix potential $V$. Throughout this paper we assume that the potential is bounded and symmetric, $V\in L^{\infty}([a,b],\R^{n\times n}),\ V^{\top}={V}$. For the operator $H_{\theta}$ we introduce the counting function $N(\cdot,\theta)$, that is, we denote the number of its eigenvalues smaller than $r$ by $N(r,\theta)$,
\begin{equation}\lb{1.4}
N(r,\theta):=\sum\nolimits_{\lambda<r}\dim _{\C}\ker (H_\theta-\lambda).
\end{equation}
We remark that $N(r,\theta)$ is well defined. Indeed, $V$ is a relatively compact perturbation of $(-\partial_x^2)_{\theta}$, thus $H_{\theta}$ has compact resolvent and is bounded from below for each $\theta$. Moreover, the boundedness of $V$ also implies that 
\begin{equation}\lb{1.5}
	\lambda_{\infty}:=\inf\nolimits_{\theta\in [0,2\pi)}\min\{\lambda:\lambda\in \Sp(H_{\theta})\}>-\infty.
\end{equation}

Next we turn to the eigenvalue problem \eqref{1.6},\eqref{1.7}. We recast the existence of non-zero solutions to \eqref{1.6},\eqref{1.7} in terms of intersections of two families, $F^1_{\theta}$ and $F^2_{\lambda}$, of $2n$ dimensional subspaces of $\C^{4n}$. Namely, the first family is determined by the boundary conditions \eqref{1.7} and is defined by
\begin{equation}\lb{1.8}
F^1_{\theta}:=\{(P, e^{\bfi \theta} P, -Q, e^{\bfi \theta}Q)^{\top}: P,Q\in \C^{n}\},\ \ \theta\in[0,2\pi),	
\end{equation}   
the second one is given by the traces of solutions to \eqref{1.6} and is defined by
\begin{equation}\lb{1.9}
F^2_{\lambda}:=\{(u(a), u(b), -u'(a), u'(b))^\top: -u''+Vu=\lambda u\},\ \ \lambda\in\R.	
\end{equation} 
Then, obviously, $\ker (H_{\theta}-\lambda)\not=\{0\}$ if and only if $F^1_{\theta}\cap F^2_{\lambda}\not=\{0\}$. As we will see below, the complex subspaces $F^1_{\theta},F^2_{\lambda}$ give rise to the real Lagrangian planes in $\Lambda(4n)$, thus allowing us to use tools from symplectic geometry. Combining this with some geometric properties of the Maslov index (mainly its homotopy invariance), we will be able to relate $N(r,\theta_1)$ and $N(r,\theta_2)$ through the Maslov index of a certain path in $\Lambda(4n)$. Furthermore, varying a parameter obtained by rescaling the operator to a smaller segment, we will compute the Maslov index, thus providing an alternative proof of the results obtained in \cite{JLM}, and estimate the Maslov index using the boundary conditions \eqref{1.7}, when $\theta$ is the varying parameter. 

Having outlined the main idea, we now switch to a more technical discussion. Our first objective is to recall from \cite{F04}(cf., \cite{BbF95}) the definition of the Maslov index, $\mi(\Upsilon,\cX)$,  for a continuous path $\Upsilon\in C([c,d],\Lambda_{\omega}(m))$; here $m\geq 1$, and we denote by $\Lambda_{\omega}(m)$ the metric space of the Lagrangian planes in $\R^{2m}$ with respect to a symplectic bilinear form $\omega$, and $\cX$ is a given Lagrangian plane in $\Lambda_{\omega} (m)$ so that $\dim \cX=m$ and $\omega$ vanishes on $\cX$. Given a subspace $\cY\subset \R^{2m}$ we denote by $P_{\cY}$ the orthogonal projection onto $\cY$. Then, following \cite[Section 2.4]{F04}, we introduce the Souriau map $S_{\cX}$ associated with the given Lagrangian plane $\cX$:
 \begin{align}
& S_{\cX}:\Lambda_{\omega}(m)\rightarrow U(\R^{2m}_{\omega}),\ S_{\cX}(\cY):= (I_{2m}-2P_\cY)(2P_{\cX}-I_{2m}),\no
\end{align}
where $U(\R^{2m}_{\omega})$ denotes the set of unitary operators on the complex space $\R^{2m}_{\omega}$. The complex vector space $\R^{2m}_{\omega}=\cX\oplus\cX^{\perp}=\cX\oplus\Omega\cX=\cX\otimes\C$ is defined via the given complex structure $\Omega$, that is, the operator satisfying
$\Omega^2=-I_{2m}, \Omega^{\top}=-\Omega, \omega(u,v)=(u,\Omega v)_{\R^{2m}}$, see \cite[Section 2.4]{F04}. For a vector $u\in \R^{2m}$, we write $u_1:=P_{\cX}u, u_2:=u-P_{\cX}u$, and define the multiplication in $\R^{2m}_{\omega}$ by
\begin{equation}\lb{1.11}
(\alpha+\bfi \beta)u:=\alpha u_1-\beta u_2+\Omega(\alpha u_2+\beta u_1),\ \alpha,\beta\in\R,
\end{equation}  
and the complex scalar product on $\R^{2m}_{\omega}$ by
\begin{equation}\lb{1.12}
(u,v)_{\omega}:=(u,v)_{\R^{2m}}-{\bfi} \omega(u,v),\ \ u,v\in \R^{2m}.
\end{equation}  
We remark that the right hand sides of \eqref{1.11},\eqref{1.12} do not depend on the choice of the Lagrangian plane $\cX$. The following property of the Souriau map from \cite[Proposition 2.52]{F04} gives rise to the spectral flow argument essential for the definition of the Maslov index of the flow $\Upsilon(\cdot)$ relative to the subspace $\cX\in\Lambda(m)$:
\begin{equation}\lb{1.12.1}
\dim_{\R}(\cX \cap \cY)= \dim_{\C}\ker(S_{\cX}(\cY)+I_{2m}),\ \cX,\cY\in \Lambda(m).
\end{equation}
Setting $\upsilon:t\mapsto S_{\cX}(\Upsilon(t))$ for $t\in[c,d]$, the Maslov index of $\Upsilon$ is defined as the spectral flow through the point $-1$ of the spectra of the family $\upsilon$ of the unitary operators in $\R^{2m}_{\omega}$. To proceed with the definition, note that there exists a partition $c=t_0<t_1<\cdots<t_N=d$ of $[c,d]$ and positive numbers $\varepsilon_j\in(0,\pi)$, such that  $\e^{\bfi(\pi+\varepsilon_j)}\not \in \Sp (\upsilon(t))$ for each $1\leq j\leq N$, see \cite[Lemma 3.1]{F04}. For any $\varepsilon>0$ and $t\in[c,d]$ we let $k(t,\varepsilon):=\sum\nolimits_{0\leq \alpha\leq \varepsilon}\ker(\upsilon(t)-\e^{\bfi(\pi+\alpha)})$ and define the Maslov index
\begin{equation}\lb{dfnMInd}
\text{Mas}(\Upsilon,\cX):=\sum\nolimits_{j=1}^{N}\left(k(t_j,\varepsilon_j)-k(t_{j-1},\varepsilon_j)\right),
\end{equation}
see \cite[Definition 3.2]{F04}. By \cite[Proposition 3.3]{F04} the number Mas$(\Upsilon,\cX)$ is well defined, i.e., it is independent on the choice of the partition $t_j$ and $\varepsilon_j$. 

The Maslov index can be computed via crossing forms. Indeed, given  $\Upsilon\in C^1([c,d], \Lambda_{\omega}(m))$ and a crossing $t_*\in[a,b]$ so that $\Upsilon(t_*)\cap\cX\not=\emptyset$, there exists a neighbourhood $\Sigma_0$ of $t_*$ and $R_t\in C^1(\Sigma_0, \cL(\Upsilon(t_*), \Upsilon(t_*)^{\perp}))$, such that $\Upsilon(t)=\{u+R_tu\big| u\in \Upsilon(t_*)\}$, for $t\in \Sigma_0$, see \cite[Lemma 3.8]{CJLS} . We will use the following terminology from \cite[Definition 3.20]{F04}.
\begin{definition}\label{def21} Let $\cX$ be a Lagrangian subspace and $\Upsilon\in C^1([c,d], \Lambda_{\omega}(m))$.
	
	({\it i}) We call $t_*\in[c,d]$ a conjugate point or crossing if $\Upsilon(t_*)\cap \cX\not=\{0\}$.
	
	({\it ii}) The quadratic form $$\cQ_{t_*,\cX}(u,v):=\frac{d}{dt}\omega_{\cH}(u,R_tv)\big|_{t=t_*}=\omega_{\cH}(u, \dot{R}_{t=t_*}v), \text{\ for\ }u,v \in \Upsilon(t_*)\cap \cX,$$  is called the crossing form at the crossing $t_*$.
	
	({\it iii}) The crossing $t_*$ is called regular if the form $\cQ_{t_*,\cX}$ is non-degenerate, positive if $\cQ_{t_*,\cX}$ is positive definite, and negative if $\cQ_{t_*,\cX}$ is negative definite.
\end{definition}	
	\begin{theorem} \cite[ Corollary 3.25]{F04}\lb{masform}
		If $t_*$ is a regular crossing of a path $\Upsilon\in C^1([c,d], \Lambda_{\omega}(m))$ then there exists $\delta>0$ such that
		
		$($i$)$ $\mi(\Upsilon_{|t-t_*|<\delta},\cX)=\text{\rm{sign}} \cQ_{t_*,\cX}$, if $t_*\in (c,d)$,
		
		$($ii$)$ $\mi(\Upsilon_{0\leq t\leq \delta},\cX)=-n_{-}(\cQ_{t_*,\cX})$, if $t_*=c$,
		
		$($iii$)$ $\mi(\Upsilon_{1-\delta\leq t\leq 1}, \cX)=n_{+}(\cQ_{t_*,\cX})$, if $t_*=d$.
	\end{theorem}
We will now review the Maslov index for {\it two} paths with values in $\Lambda_{\omega}(m)$, see \cite[Section 3.5]{F04}. Let us fix $\Upsilon_1,\Upsilon_2\in C([c,d], \Lambda_{\omega}(m))$, and introduce $\Delta:=\{(p,p):p\in \R^{2m}\}$, the diagonal plane.  On  $\R^{2m}\oplus\R^{2m}$ we define the symplectic form $\wti{\omega}:=\omega\oplus(-\omega)$ with the complex structure $\wti{\Omega}:=\Omega\oplus(-\Omega)$, denoting the resulting space of Lagrangian planes  by $\Lambda_{\wti{\omega}}(2m)$. We consider the path $\wti{\Upsilon}:=\Upsilon_1\oplus\Upsilon_2\in C ([c,d],\Lambda_{\wti{\omega}}(2m))$ and define the Maslov index of the two paths $\Upsilon_1,\Upsilon_2$ as $\mi(\Upsilon_1,\Upsilon_2):=\mi(\wti{\Upsilon},\Delta).$
If $\Upsilon_2(t)=\cX$ for all $t\in[c,d]$, then $ \mi(\Upsilon_1\oplus \Upsilon_2,\Delta)=\mi(\Upsilon_1,\cX).$
\section{Variation Of The Parameter $\theta$} \lb{SectionTheta}
In this section, for the counting function of the Schr\"odinger operator $H_{\theta}$ equipped with a $\theta-$periodic boundary conditions, we derive a relation between $N(r,\theta_2)-N(r,\theta_1)$ and the Maslov index of a path associated with the eigenvalue problem \eqref{1.6},\eqref{1.7}. In order to use the symplectic approach in counting eigenvalues, we rewrite equation \eqref{1.6} and the boundary conditions \eqref{1.7} in terms of the real and imaginary parts of $u$ and arrive at the $(2n\times2n)$ system
\begin{align}
&-y''+(V\otimes I_2)y=\lambda y,\ \lambda\in\R,\ y:[a,b]\rightarrow \R^{2n}\lb{2.1},\\
&\quad y(b)=(I_n\otimes M_{\theta})y(a),\ y'(b)=(I_n\otimes M_{\theta})y'(a),\lb{2.2}
\end{align} 
where the components of the vectors $y=(y_k)_{k=1}^{2n}$ in \eqref{2.1} and $u=(u_k)_{k=1}^{n}$ in \eqref{1.6} are related via  
\begin{equation}\lb{2.4}
y_{2k-1}=\Re u_k, y_{2k}=\Im u_k,\ 1\leq k\leq n.
\end{equation}
The number of the linearly independent in $L^2([a,b],\R^{2n})$ solutions to \eqref{2.1},\eqref{2.2} is equal to $2\dim_{\C}\ker(H_{\theta}-\lambda)$. Indeed if $u\in \ker(H_{\theta}-\lambda)$, then $y$ from \eqref{2.4} and $(-I_n\otimes J)y$ are solutions to \eqref{2.1},\eqref{2.2}.
Moreover, linearly independent solutions $u$ in $L^2([a,b],\C^{n})$ give linearly independent solutions $y$ and $(-I_n\otimes J)y$ in $L^2([a,b],\R^{2n})$. Conversely, as we already observed, solutions to \eqref{2.1},\eqref{2.2} appear in pairs, $y$ and $(-I_n\otimes J)y $, so that by mapping them into $u, \bfi u$, we obtain a linearly dependent in $L^2([a,b],\C^n)$ pair of solutions to \eqref{1.6},\eqref{1.7}. 

Next, we introduce a symplectic bilinear form on $\R^{8n}$, $\omega:\R^{8n}\times \R^{8n}\rightarrow \R,$ $\omega(p,q):=(p, (J\otimes I_{4n})q)_{\R^{8n}}$, and the planes $\cF_{\theta}^1$ and $\cF_{\lambda}^2$ associated with equation \eqref{2.1} and boundary conditions \eqref{2.2},
\begin{align}
&\cF^1_{\theta}:=\left\{(p,(I_{n}\otimes M_{\theta})\ p,-q,(I_{n}\otimes M_{\theta})\ q)^{\top}: p,q\in \R^{2n}\right\},\ \lb{2.5}\\
&\cF^2_{\lambda}:=\{\tr(y): y\text{\ are solutions to \eqref{2.1}}\}; \lb{2.6}
\end{align}
and the trace $\tr(y):=(y(a),y(b),-y'(a),y'(b))^{\top}$.
\begin{proposition}\lb{p2.1}
$($i$)$ For each $\theta\in[0,2\pi)$ one has $\cF^1_{\theta}\in \Lambda_{\omega}(4n)$.

$($ii$)$ For each $\lambda\in\R$ one has $\cF^2_{\lambda}\in \Lambda_{\omega}(4n)$.
\end{proposition}
\begin{proof}
({\it i}) Pick any two vectors $h_{\ell}\in \cF_{\theta}^1$, that is, for some $p_{\ell},q_{\ell}\in \R^{2n}$ let\newline $h_{\ell}=\left(p_{\ell}, (I_{n}\otimes M_{\theta}) p_{\ell},-q_{\ell}, (I_{n}\otimes M_{\theta}) \ q_{\ell}\right)^{\top},\ {\ell}=1,2.$
Then
\begin{align}
\omega(h_1,h_2)=&(p_1,- q_2)_{\R^{2n}}+\left((I_{n}\otimes M_{\theta})\ p_1,(I_{n}\otimes M_{\theta})\ q_2\right)_{\R^{2n}}\no\\
&+(-q_1,- p_2)_{\R^{2n}}+\left((I_{n}\otimes M_{\theta})\ q_1, (-I_{n}\otimes M_{\theta})\ p_2\right)_{\R^{2n}}\no\\
=&-(p_1, q_2)_{\R^{2n}}+\left(p_1,(I_{n}\otimes M_{\theta}^{\top})\ (I_{n}\otimes M_{\theta})\ q_2\right)_{\R^{2n}}\no\\
&+(q_1,p_2)_{\R^{2n}}-\left(\ q_1, {(I_{n}\otimes M_{\theta}^{\top})}\ (I_{n}\otimes M_{\theta})\ p_2\right)_{\R^{2n}}=0,\lb{2.8}
\end{align}
since $(I_{n}\otimes M_{\theta}^{\top})\ (I_{n}\otimes M_{\theta})=I_{2n}$. The fact that $\dim \cF_{\theta}^1=4n$ is apparent.

({\it ii}) Pick any two vectors $Y_{\ell}\in \cF_{\lambda}^2$, so that for some solutions $y_{\ell}$ of \eqref{2.1},\newline $Y_{\ell}=\left(y_{\ell}(a),y_{\ell}(b),-y'_{\ell}(a),y'_{\ell}(b)\right)^{\top},\ {\ell}=1,2.$ 
Then integration by parts yields
\begin{align*}
\omega(Y_1,Y_2)=&\int_a^b(-y_1'',y_2)_{\R^{2n}}-(y_1,-y_2'')_{\R^{2n}}dx\no\\
=&\int_a^b(\lambda y_1-(V\otimes I_2) y_1,y_2)_{\R^{2n}}-(y_1,\lambda y_2-(V\otimes I_2)y_2)_{\R^{2n}}dx=0.
\end{align*}
The equation $\omega(Y_1,Y_2)=0$ together with $\dim \cF_{\lambda}^2=4n$ prove the assertion. 
\end{proof}
In order to formulate our principal result we need to define two continuous and piecewise differentiable paths with values in $\Lambda_\omega(4n)$. For any fixed $\theta_1,\theta_2\in[0,2\pi)$ and $r\in \R$ larger than $\lambda_{\infty}$ from \eqref{1.5} we introduce a parametrization of the rectangle displayed in Figure 1(I) as follows.
Let  $\Gamma=\cup_{j=1}^4\Gamma_j$ be the boundary of the square $\{(\theta,\lambda): \theta\in[\theta_1,\theta_2], \lambda\in[\lambda_{\infty},r] \}$, let $\Sigma=\cup_{j=1}^4\Sigma_j$, and let $\Sigma\ni s\mapsto (\theta(s),\lambda(s))\in \Gamma$ be the parametrization of $\Gamma$ defined by
\begin{align}
&\lambda(s)=s,\, \theta(s)=\theta_1 ,\, s\in\Sigma_1:=[\lambda_{\infty},r],\lb{par2}\\
&\lambda(s)=r,\, \theta(s)=s+\theta_1-r ,\, s\in\Sigma_2:=[r,r+\theta_2-\theta_1 ],\lb{par3}\\
&\lambda(s)= -s+2r+\theta_2-\theta_1 ,\, \theta(s)= \theta_2,\no\\ &\hskip3cm s\in\Sigma_3:=[r+\theta_2-\theta_1 ,2r+\theta_2-\theta_1-\lambda_{\infty}],\lb{par4}\\
&\lambda(s)=\lambda_{\infty},\, \theta(s)=-s+2r+2\theta_2-\theta_1 -\lambda_{\infty},\lb{par5}\\
&\hskip3cm s\in\Sigma_4:=[2r+\theta_2-\theta_1-\lambda_{\infty},2r+ 2(\theta_2-\theta_1)-\lambda_{\infty}].\no
\end{align}

Let us recall the Lagrangian planes $\cF_{\theta}^1,\cF_{\lambda}^2$ from \eqref{2.5},\eqref{2.6}, and the eigenvalue counting function $N(r,\theta)$ from \eqref{1.4}.
\begin{theorem}\lb{t2.2}
	Let $V\in L^{\infty}([a,b],\R^{n\times n})$ and $\ V=V^{\top}$. If $0\leq \theta_1<\theta_2< 2\pi$ and $r>\lambda_{\infty}$, then
	\begin{equation}\lb{2.15}
		N(r,\theta_2)-N(r,\theta_1)={}^{1}/_{2} \mi(\cF_{\theta}^1|_{\theta_1\leq\theta\leq \theta_2},\cF_{r}^2).
	\end{equation}
\end{theorem}
\begin{proof} Given parametrization \eqref{par2}-\eqref{par5}, we introduce the paths $\Upsilon_1(s):=\cF_{\theta(s)}^1$, $\Upsilon_2(s):=\cF_{\lambda(s)}^2$, and their direct sum $\wti{\Upsilon}(s):=\Upsilon_1(s)\oplus\Upsilon_2(s),\ s\in \Sigma$, taking values in $\Lambda_{\tilde{\omega}}(8n)$ for $\wti{\omega}:=\omega\oplus(-\omega)$. Since $\wti{\Upsilon}$ is a closed loop, we have $\mi (\wti{\Upsilon}(s),\Delta)=0$, where $\Delta:=\{(p,p):p\in \R^{8n}\}$. On the other hand, 
	\begin{align}
		\mi(\wti{\Upsilon}(s),\Delta)=&\mi(\wti{\Upsilon}(s)|_{\Sigma_1},\Delta)+\mi(\wti{\Upsilon}(s)|_{\Sigma_2},\Delta)\no\\
		&+\mi(\wti{\Upsilon}(s)|_{\Sigma_3},\Delta)+\mi(\wti{\Upsilon}(s)|_{\Sigma_4},\Delta)\lb{2.16}.
	\end{align}
We will compute each term individually and use  \eqref{2.16} to obtain formula \eqref{2.15}.

{\bf Step 1.} Since $\theta(s)=\theta_1$ for all $s\in\Sigma_1$, one has $\wti{\Upsilon}=\Upsilon_1(s)\oplus\Upsilon_2(s)=\cF_{\theta_1}^1\oplus \Upsilon_2(s)$, thus $\mi(\wti{\Upsilon}|_{\Sigma_1},\Delta)=-\mi(\Upsilon_2(s),\cF_{\theta_1}^1)$.

Let $s_*\in (\lambda_{\infty}, r)$ be a conjugate point, i.e. $\Upsilon_2(s_*)\cap \cF_{\theta_1}^1\not=\{0\}$. There exists a small neighbourhood $\Sigma_{s_*}\subset(\lambda_{\infty}, r)$ of $s_*$ and a family \[(s+s_*)\mapsto R_{(s+s_*)}\,\text{ in }\, C^1\big(\Sigma_{s_*}, \cL(\Upsilon_2(s_*), \Upsilon_2(s_*)^{\perp})\big),\ R_{s_*}=0_{8n},\] such that $\Upsilon_2(s)=\{u+R_{(s+s_*)}u\big| u\in \Upsilon_2(s_*) \}$ for all $(s+s_*)\in \Sigma_{s_*}$ (cf., the discussion prior Definition \ref{def21}). Let us fix a solution $y^0$ to \eqref{2.1},\eqref{2.2} with $\lambda=\lambda(s_*)$ and $\theta=\theta_1$ (this solution exists since $s_*$ is a conjugate point). Then $\tr(y^0)+R_{(s+s_*)}\tr(y^0)\in\Upsilon_2(s)=\cF^2_{\lambda(s)}$ for small $|s|$, and thus there exists a family of solutions $y_s^0$ of \eqref{2.1} such that $\tr(y^0_s):=\tr(y^0)+R_{(s+s_*)}\tr(y^0)$. Next we calculate the crossing form using integration by parts and that $y_s^0$ solves \eqref{2.1} with $\lambda=\lambda(s_*+s)$:
\begin{align}
\omega&(\tr(y^0), R_{(s+s_*)}\tr(y^0))\no\\
&=\int_a^b(-{y^0}'',y_s^0-y^0)_{\R^{2n}}-(y^0,-({y}^0_s-y^0)'')_{\R^{2n}}dx\no\\
&=\int_a^b(y^0,{y^0_s}''-(V\otimes I_2)y^0_s+\lambda(s_*)y_s^0)_{\R^{2n}}dx\no\\
&=(\lambda(s_*)-\lambda(s+s_*))\int_a^b(y^0,y_s^0)_{\R^{2n}}dx=-s\int_a^b(y^0,y_s^0)_{\R^{2n}}dx.\lb{2.17}
\end{align}
Differentiating with respect to $s$ at $s=0$ yields
\begin{align}
\cQ_{s_*,\cF_{\theta_1}^1 }(\tr(y^0),\tr(y^0))&:=\frac{d}{ds}\omega(\tr(y^0), R_{(s+s_*)}\tr(y^0))\big|_{s=0}=-\|y^0\|^2_{L^2([a,b],\R^{2n})}.\no
\end{align}
By Theorem \ref{masform} ({\it i}) we therefore have 
\begin{align}
&\mi\big(\Upsilon_{2}\big|_{\Sigma_{s_*}},\cF_{\theta_1}^1\big)=\sign\ \cQ_{s_*,\cF_{\theta_1}^1}=-\dim_{\R}\big(\Upsilon_2(s_*)\cap \cF_{\theta_1}^1\big)\lb{2.18}\\
&=-\#\left\{\begin{matrix}
\text{\  linearly independent in $L^2([a,b],\R^{2n})$\ }\\
\text{\  solutions to \eqref{2.1},\eqref{2.2}\ }\\
\text{with $\lambda=\lambda(s_*)$ and $\theta=\theta_1$}
\end{matrix}\right\}=-2\dim_{\C}\ker(H_{\theta_1}-\lambda(s_*)).\no
\end{align}
Formula \eqref{2.18} holds for all crossings $s_*\in \Sigma_1$, thus, using \eqref{par2},
\begin{align}
\mi\big(\Upsilon_{2}\big|_{\Sigma_1}, \cF_{\theta_1}^1 \big)&=\sum\nolimits_{\substack{\lambda_{\infty}<s<r:\\
		\Upsilon_2(s)\cap \cF_{\theta_1}^1\not=\{0\}	}}\sign\ \cQ_{s,\cF_{\theta_1}^1}+n_+\left(\cQ_{r,\cF_{\theta_1}^1}\right)\no\\
&=-\sum\nolimits_{\lambda_{\infty}<s<r}2\dim_{\C}\ker(H_{\theta_1}-\lambda(s))=-2N(r,\theta_1),\lb{2.12}
\end{align}
where we used \eqref{2.18} and that $\cQ_{r,\cF_{\theta_1}^1}$ is negative definite, that is,  $n_+(\cQ_{r,\cF_{\theta_1}^1 })=0$, where the last equality holds due to \eqref{1.5}.
Finally,  
	$\mi(\wti{\Upsilon}|_{\Sigma_1},\Delta)=2N(r,\theta_1)$.

{\bf Step 2.} Calculations similar to the ones  in Step 1 lead to 
$\mi(\wti{\Upsilon}|_{\Sigma_3},\Delta)=-2N(r,\theta_2)$.
Indeed, the key fact that was used in Step 1 and can be employed here is that $\theta(s)$ is a constant and $|\lambda(s+s_*)-\lambda(s_*)|=|s|$ for all $s\in\Sigma_1\cup\Sigma_3$.

{\bf Step 3.} Since $\Sp( H_{\theta})\subset (\lambda_{\infty},+\infty)$ by \eqref{1.5}, 
one has
$\mi(\wti{\Upsilon}|_{\Sigma_4},\Delta)=0$.

{\bf Step 4.} Combining this, \eqref{2.16}, 
$\mi(\wti{\Upsilon}|_{\Sigma_2},\Delta)=\mi(\cF_{\theta}^1|_{\theta_1\leq\theta\leq \theta_2},\cF_{r}^2)$ and $\mi(\wti{\Upsilon},\Delta)=0$, one obtains \eqref{2.15}. 
\end{proof}
Let us define the counting function for an interval, 
\begin{equation}
	N([r_1,r_2),\theta):=\sum\nolimits_{r_1\leq \lambda<r_2}\dim _{\C}\ker (H_\theta-\lambda),\ r_1<r_2,\ \theta\in[0,2\pi).\no
\end{equation}
\begin{corollary}
Under the assumptions of Theorem \ref{t2.2} one has
	\begin{align}\lb{3.40}
	N([r_1,& r_2),\theta_2)-N([r_1,r_2), \theta_1)\no\\
	&={}^{1}/_{2} \mi(\cF_{\theta}^1|_{\theta_1\leq\theta\leq \theta_2},\cF_{r_2}^2)-{}^{1}/_{2} \mi(\cF_{\theta}^1|_{\theta_1\leq\theta\leq \theta_2},\cF_{r_1}^2),\ r_1<r_2.
	\end{align}
\end{corollary}
\begin{proof}
	The proof is analogous to that of Theorem \ref{t2.2} with $\lambda_{\infty}:=r_1$ and $r:=r_2$.  	
\end{proof}
\begin{corollary}\lb{cor3.4}
If $V\in L^{\infty}([a,b],\R^{n\times n})$ then, for all $r$ and $r_1<r_2$, 

$(${\it i}$)$ if $0\leq\theta_1<\theta_2<2\pi$ then
\begin{align}\lb{2.25}
|N(r,\theta_2)-N(r,\theta_1)|&\leq 2n,
\\
\lb{3.41}
|N([r_1, r_2),\theta_2)-N([r_1,r_2), \theta_1)|&\leq 4n,
\end{align}

$(${\it ii}$)$ if $0<\theta_1<\theta_2<\pi$ or  $\pi<\theta_1<\theta_2<2\pi$ then
\begin{align}
|N(r,\theta_2)-N(r,\theta_1)|&\leq n,\lb{2.24}\\
|N([r_1, r_2),\theta_2)-N([r_1,r_2), \theta_1)|&\leq 2n.\lb{3.42}
\end{align}

\end{corollary}
\begin{proof}
First we notice that 
\begin{align}
&|{}^{1}/_{2}\mi(\cF_{\theta}^1|_{\theta_1\leq\theta\leq \theta_2},\cF_{r}^2)|\no\\
&={}^{1}/_{2}\Big|-n_-(\cQ_{\theta_1,\cF_{r}^2})+\sum\nolimits_{\substack{\theta_1<s<\theta_2:\\
		\Upsilon_2(s)\cap \cF_{\theta_1}^1\not=\{0\}	}}\sign\ \cQ_{s,\cF_{r}^2}+n_+(\cQ_{\theta_2,\cF_{r}^2})\Big|\no\\
&\leq {}^{1}/_{2}\sum\limits_{\substack{\theta_1\leq s\leq \theta_2}}\#\left\{\begin{matrix}
\text{\ linearly independent in $L^2([a,b],\R^{2n})$\ }\\
\text{\  solutions to \eqref{2.1},\eqref{2.2}\ }\\
\text{with $\lambda=r$ and $\theta=s$}\end{matrix}\right\}\no\\&=\sum\nolimits_{{\theta_1\leq s\leq\theta_2}}\dim_{\C}\ker(H_{s}-r).\lb{2.26}
\end{align}
There exist $2n$ linearly independent in $L^2([a,b],\C^n)$ solutions to \eqref{1.6}, thus
\begin{equation}\lb{2.27}
\sum\nolimits_{{\theta_1\leq s\leq\theta_2}}\dim_{\C}\ker(H_{s}-r)\leq 2n,
\end{equation}
implying \eqref{2.25}, which in turn leads to \eqref{3.41}.
In order to show ({\it ii}) we prove that 
\begin{equation}\lb{2.28}
\sum\nolimits_{{0<\theta_1\leq s\leq\theta_2<\pi}}\dim_{\C}\ker(H_{s}-r)\leq n.
\end{equation}
Let us assume that we can fix $n$ values of the parameter $\theta$, enumerated in nondecreasing order $0<\theta_1\leq \vartheta_1\leq\dots\leq \vartheta_n \leq\theta_2<\pi$ and $n$ linearly independent in $L^2([a,b],\C^n)$ functions $u_k$, so that $u_k\in\ker(H_{\vartheta_k}-r)$, that is  $u_k$ is a solution to \eqref{1.6},\eqref{1.7} with $\lambda=r$ and $\theta=\vartheta_k$. If this assumption is not satisfied then \eqref{2.28} holds automatically. We note that $\overline{u_k}\in\ker(H_{-\vartheta_k}-r), 1\leq k \leq n$. 

We claim that the system of vectors $u_1,\overline{u_1}\dots, u_n,\overline{u_n}$ is linearly independent in $L^2([a,b],\C^n)$. Indeed, pick $\alpha_1,\dots, \alpha_n, \beta_1,\dots, \beta_n$ such that 
\begin{equation}\lb{2.29}
U_0(x):=\sum\nolimits_{j=1}^{n}\alpha_{j} u_{j}(x)+\sum\nolimits_{k=1}^{n}\beta_{k} \overline{u_k(x)}=0,\ x\in[a,b].
\end{equation}
For $1\leq l \leq n$ we define 
\begin{equation}
	U_l:=\sum\nolimits_{j=1}^{n}\prod\nolimits_{i=1}^{l}(e^{-\bfi\vartheta_i}-e^{\bfi\vartheta_j})\alpha_{j} u_{j}+\sum\nolimits_{k=1}^{n}\prod\nolimits_{i=1}^{l}(e^{-\bfi\vartheta_i}-e^{-\bfi\vartheta_k})\beta_{k} \overline{u_k}. \lb{d1}
\end{equation}
Our immediate objective is to prove that $U_l=0$ for each $l$. We use induction in $l$. 
For the base case $l=1$, evaluating $U_0(a), U_0(b)$ from \eqref{2.29} and using the boundary conditions satisfied by $u_k,\overline{u}_k$, we arrive at
\begin{equation}\lb{2.29.1}
U_0(a)=\sum\nolimits_{j=1}^{n}\alpha_{j} u_j(a)+\sum\nolimits_{k=1}^{n}\beta_{k} \overline{u_k(a)}=0,
\end{equation}
\begin{equation}\lb{2.29.2}
U_0(b)=\sum\nolimits_{j=1}^{n}\alpha_{j} e^{\bfi \vartheta_j}u_{j}(a)+\sum\nolimits_{k=1}^{n}\beta_{k} e^{-\bfi \vartheta_k}\overline{u_k(a)}=0.
\end{equation}
Multiplying \eqref{2.29.1} by $e^{-\bfi \vartheta_1}$ and subtracting \eqref{2.29.2} yields $U_1(a)=0$. Similarly one obtains  $U'_1(a)=0$, and since $U_1$ is a solution to the second order differential equation \eqref{1.6}, it is identically equal to zero. For the inductive step, we assume that $U_l\equiv 0$, then $e^{-\bfi \vartheta_{l+1}}U_l(a)-U_l(b)=0$ and $e^{-\bfi \vartheta_{l+1}}U'_l(a)-U'_l(b)=0$, and using the boundary conditions satisfied by $u_k,\overline{u}_k$ we conclude that $U_{l+1}(a)=U'_{l+1}(a)=0$. Since, in addition, $U_{l+1}$ solves the ODE it must vanish on $[a,b]$, completing the proof of $U_l=0$ for all $1\leq l \leq n$. The second term in \eqref{d1} vanishes if $l=n$, so\begin{equation}\lb{d2}
	U_n=\sum\nolimits_{j=1}^{n}\prod\nolimits_{i=1}^{n}(e^{-\bfi\vartheta_i}-e^{\bfi\vartheta_j})\alpha_{j} u_{j}.
\end{equation}
Since the system $u_1,\dots,u_n$ is linearly independent, $\alpha_j=0$. Using this in \eqref{2.29} similarly yields $\beta_k=0$, proving the claim.

 Next, to complete the proof, that is, to confirm \eqref{2.28} and therefore \eqref{2.24} and \eqref{3.42}, we show that if $u\in\ker(H_{{\theta}}-r)$ for some ${\theta}\in (0,\pi)$, then ${u}\in {\rm span}\{u_1,\dots, u_n\}$. As we just showed, $\{u_1,u_2,\dots,u_n, \overline{u_1}, \dots\overline{u_n}\}$ is the fundamental system of solutions to \eqref{1.6}. Then for some $\mu_1,\dots\mu_n,\nu_1,\dots,\nu_n$, one has
\begin{equation}\lb{3.37}
u=\sum\nolimits_{k=1}^{n}\mu_ku_k+\nu_k\overline{u_k}.
\end{equation}
Using ${u}(b)=e^{\bfi {\theta}}{u}(a),{u}'(b)=e^{\bfi {\theta}}{u}'(b),$ we obtain
\begin{equation}\lb{3.38}
\sum\nolimits_{k=1}^{n} (e^{\bfi {\theta}}-e^{\bfi \vartheta_k}) \mu_ku_k(a)+(e^{\bfi {\theta}}-e^{-\bfi \vartheta_k}) \nu_k\overline{u_k(a)}=0,
\end{equation}
\begin{equation}\lb{3.39}
\sum\nolimits_{k=1}^{n} (e^{\bfi {\theta}}-e^{\bfi \vartheta_k}) \mu_ku'_k(a)+(e^{\bfi {\theta}}-e^{-\bfi \vartheta_k}) \nu_k\overline{u'_k(a)}=0, 
\end{equation}
thus $(e^{\bfi {\theta}}-e^{\bfi \vartheta_k}) \mu_k=0,(e^{\bfi {\theta}}-e^{-\bfi \vartheta_k}) \nu_k=0,\ 1\leq k\leq n.$ Recalling that $\vartheta_1,\cdots, \vartheta_n,  \theta\in (0,\pi)$, one infers $\nu_1=\dots=\nu_n=0$ and ${u}\in {\rm span}\{u_1,\dots, u_n\}$ as asserted.
\end{proof}
In the concluding part of this section we discuss the monotonicity properties of the curves of {\it non-degenerate} or {\it simple} eigenvalues of $H_{\theta}$. Any such curve $\lambda(\theta)$ is analytic in $(\theta_1,\theta_2)\subset (0,\pi)\cup(\pi,2\pi)$, since the operator $\left(-\partial_x^2\right)_{\theta}$ is (and hence so is $H_{\theta}$), see, e.g., \cite[XIII.16]{RS78}. Our argument relies on the following important relation between $\dot{\lambda}(\theta(s_*))$ and the crossing form $\cQ_{s_*,\cF^2_{\lambda(s_*)}}$ at the crossing point $s_*$ (here and bellow we let $\dot{\lambda}:=\frac{d\lambda}{d\theta}$ ). 
\begin{theorem}\lb{FE} Let  $\lambda(\theta)\in\Sp(H_{\theta})$ be a simple eigenvalue with the corresponding eigenfunction $u_{\theta}$ for  $\theta\in(\theta_1,\theta_2)\subset (0,\pi)\cup(\pi,2\pi)$. Assume that $s_*$ is a crossing of the path $\cF_{\theta(s)}^1$ and the fixed Lagrangian plain $\cF_{\lambda(s_*)}^2$. Then one has
\begin{equation}\lb{a1}
	\frac{d\lambda}{d\theta}(\theta(s_*))=2\Im(u'_{\theta}(a),u_{\theta}(a))_{\C^n}=\cQ_{s_*,\cF_{\lambda(s_*)}^2}(\tr y_{\theta(s_*)},\tr y_{\theta(s_*)});
\end{equation} 
where $y_{\theta}$ are defined as $(y_{\theta})_{2k-1}:=\Re (u_{\theta})_k, (y_{\theta})_{2k}:=\Im (u_{\theta})_k,$\ $\ 1\leq k\leq n.$
\end{theorem}
\begin{proof} Since $u_{\theta}$ is the eigenfunction corresponding to the eigenvalue $\lambda(\theta)$ one has $-u''_{\theta}+Vu_{\theta}=\lambda (\theta)u_{\theta}.$
Differentiating with respect to $\theta$ yields	$-\dot{u}''_{\theta}+V\dot{u}_{\theta}=\dot{\lambda} (\theta)u_{\theta}+{\lambda} (\theta)\dot{u}_{\theta},$
multiplying by $\overline{u_{\theta}}$ and integrating one obtains
\begin{equation}\lb{a2}
(H_{\theta}\dot{u}_{\theta},u_{\theta})_{L^2([a,b],\C^n)}=(\dot{\lambda} (\theta)u_{\theta},u_{\theta})_{L^2([a,b],\C^n)}+({\lambda} (\theta)\dot{u}_{\theta},u_{\theta})_{L^2([a,b],\C^n)}.
\end{equation}
Next, we integrate by parts the first term in \eqref{a2} and arrive at
\begin{align}
(\dot{u}_{\theta},&H_{\theta}u_{\theta})_{L^2([a,b],\C^n)}-({\dot{u}'}_{\theta}(b),u_{\theta}(b))_{\C^n}+({\dot{u}'}_{\theta}(a),u_{\theta}(a))_{\C^n}
+(\dot{{u}}_{\theta}(b),u'_{\theta}(b))_{\C^n}\no\\
&-(\dot{{u}}_{\theta}(a),u'_{\theta}(a))_{\C^n}=\dot{\lambda} (\theta)\|u_{\theta}\|_{L^2([a,b],\C^n)}^2+(\dot{u}_{\theta},{\lambda} (\theta)u_{\theta})_{L^2([a,b],\C^n)}.\no
\end{align}
Hence,
\begin{align}
&\dot{\lambda} (\theta)\|u_{\theta}\|_{L^2([a,b],\C^n)}^2=-({\dot{u}'}_{\theta}(b),u_{\theta}(b))_{\C^n}+({\dot{u}'}_{\theta}(a),u_{\theta}(a))_{\C^n}\no\\
&\qquad\qquad+(\dot{{u}}_{\theta}(b),u'_{\theta}(b))_{\C^n}-(\dot{{u}}_{\theta}(a),u'_{\theta}(a))_{\C^n}.\lb{a3}
\end{align}
We differentiate  \eqref{1.7} with $u$ replaced by $u_{\theta}$ and plug the obtained values of $\dot{u}_{\theta}(b),\dot{u}_{\theta}'(b)$, expressed in terms of $\dot{u}_{\theta}(a),\dot{u}_{\theta}'(a)$, in \eqref{a3}. We arrive at
\begin{equation}\lb{b1}
\dot{\lambda} (\theta)\|u_{\theta}\|_{L^2([a,b],\C^n)}^2=2\re(u'_{\theta}(a),\bfi u_{\theta}(a))_{\C^n}=2(y'_{\theta}(a), (I_n\otimes J) y_{\theta}(a))_{\R^{2n}}.
\end{equation}
We are now ready to compute the right hand side of \eqref{a1}. To this end let us use parametrization \eqref{par3} with $r:=\lambda(s_*)$. For some small neighbourhood $\Sigma_{s_*}\subset(r,r+ \theta_2-\theta_1)$ of $s_*$ there exists a family $(s+s_*)\mapsto R_{(s+s_*)}$ such that $$ R_{(s+s_*)}\in C^1\Big(\Sigma_{s_*}, \cL\big(\cF^1_{\theta(s_*)},  (\cF^1_{\theta(s_*)})^{\perp}\big)\Big),\ R_{s_*}=0_{8n},$$ and $\cF^1_{\theta(s)}=\{h+R_{(s+s_*)}h\big| h\in \cF^1_{\theta(s_*)} \}$ for all $(s+s_*)\in \Sigma_{s_*}$. Then a family of vectors $h_s:=\tr y_{\theta(s_*)}+R_{(s+s_*)}\tr y_{\theta(s_*)}$ is of the form 
\begin{equation}\lb{b2}
h_s=\big(p_{\theta(s)}, (I_n\otimes M_{\theta(s)})p_{\theta(s)},-q_{\theta(s)},(I_n\otimes M_{\theta(s)})q_{\theta(s)} \big)^{\top},
\end{equation}
with $p_{\theta(s_*)}=y_{\theta(s_*)}(a)\text{\ and\ }q_{\theta(s_*)}=y'_{\theta(s_*)}(a).$
Denoting $\frac{df}{ds}=\grave{f}$ and differentiating \eqref{b2} with respect to $s$ near $s_*$ we obtain
\begin{align}
\grave{h}_s=\big(&\grave{p}_{\theta(s)},-\big(I_n\otimes (JM_{\theta(s)})\big)p_{\theta(s)}+(I_n\otimes M_{\theta(s)})\grave{p}_{\theta(s)},\no\\
&-\grave{q}_{\theta(s)},-\big(I_n\otimes (JM_{\theta(s)})\big)q_{\theta(s)}+(I_n\otimes M_{\theta(s)})\grave{q}_{\theta(s)} \big)^{\top}.\lb{b4}
\end{align}
Finally,
\begin{align}
&\cQ_{s_*,\cF_{\lambda(s_*)}^2}(\tr y_{\theta(s_*)},\tr y_{\theta(s_*)})=\omega(\tr y_{\theta(s_*)},\grave{R}_{s_*}\tr y_{\theta(s_*)})=\omega(\tr y_{\theta(s_*)}, \grave{h}_{s_*})\no\\&\quad=2\big(q_{\theta(s_*)},(I_n\otimes J) p_{\theta(s_*)}\big)_{\R^{2n}}=2\big(y'_{\theta(s_*)}(a),(I_n\otimes J) y_{\theta(s_*)}(a)\big)_{\R^{2n}}.\lb{b5}
\end{align}
Equation \eqref{b1} with $\theta=\theta(s_*)$ combined with \eqref{b5} and normalization of the eigenfunction $u_{\theta(s_*)}$ yield \eqref{a1}.
\end{proof}
With Theorem \ref{FE} at hands we derive monotonicity of eigenvalues with respect to $\theta$. We remark that the second part of the following claim is a well known result which can be found in \cite[Theorem XIII.89]{RS78}
\begin{proposition} \lb{mon1}
Let  $\lambda(\theta)\in\Sp(H_{\theta})$ be a simple eigenvalue with the corresponding eigenfunction $u_{\theta}$ for  $\theta\in(\theta_1,\theta_2)\subset (0,\pi)\cup(\pi,2\pi)$. Then either $\lambda(\theta)$ is monotone or for some $\theta_*\in (\theta_1,\theta_2)$ one has 
\begin{equation}\lb{c1}
\Im(u'_{\theta_*}(a),u_{\theta_*}(a))_{\C^n}=0.
\end{equation}
Moreover, if $n=1$, that is, the potential $V$ is scalar valued, then $\lambda({\theta})$ is monotone.
\end{proposition}
\begin{proof}
The alternative is a simple corollary of the first equation in \eqref{a1}.
Next, we show that if $n=1$ then $\dot{\lambda}$ does not vanish on $(\theta_1,\theta_2)$. Indeed, using \eqref{a1}, $\bfi\dot{\lambda}(\theta_{*})={2\bfi}\Im(u'_{\theta_*}(a),u_{\theta_*}(a))_{\C^n}=W(u_{\theta_*},\overline{u}_{\theta_*})$, the Wronskian of $u$ and $\overline{u}_{\theta_*}$. But $W(u_{\theta_*},\overline{u}_{\theta_*})\not=0$ by linear independence of $u_{\theta_*}$ and $\overline{u}_{\theta_*}$.
\end{proof}
We remark that the expression $\Im(u'_{\theta}(x),u_{\theta}(x))_{\C^n}$ does not depend on $x\in[a,b]$ since $u_\theta$ solves \eqref{1.6}, \eqref{1.7}. Also, this expression plays an essential role in \cite{CJ}.

\section{Variation Of The Scaling Parameter}\lb{SectionT}
Throughout this section the parameter $\theta\in[0,2\pi)$ is fixed. We consider the operator $\left(-\partial_x^2\right)_{\theta}$ on $L^2([-L,L])$, here the interval is chosen to be symmetric to simplify notations. Restricting the potential $V$ to $[-tL,tL]$, $0< t\leq 1$, we obtain a family of Schr\"odinger operators $H_{\theta}(t):=-\partial_x^2+V\big|_{[-tL,tL]}$ on $L^2([-tL,tL])$ with the $\theta-$periodic boundary conditions, that is, conditions \eqref{1.7} with $a=-tL$ and $b=tL$. Denoting $\cN(r,t):=\sum_{\lambda<r}\dim _{\C}\ker (H_\theta(t)-\lambda),$ we will compute below $\cN(r,t_2)-\cN(r,t_1)$ in terms of the Maslov index (Theorem \ref{t3.1}), and evaluate it (Theorem \ref{4.2}) in case of the sign-definite potential. Firstly, let us write the eigenvalue problem 
$H_{\theta}(t)u=\lambda u$
as the boundary value problem
\begin{align}
&-u''(x)+V(x)u(x)=\lambda u(x),\ \lambda\in\R, x\in [-tL,tL],\lb{3.53.1}\\
&\qquad\ \ u(tL)=e^{\bfi \theta}u(-tL),\ u'(tL)=e^{\bfi \theta}u'(-tL).\lb{3.54.2}
\end{align}
Secondly, rescaling equations \eqref{3.53.1},\eqref{3.54.2}, we arrive at
\begin{align}
-&u''(x)+t^2V(tx)u(x)=t^2\lambda u(x),\ \lambda\in\R, x\in [-L,L]\lb{3.53},\\
&\qquad u(L)=e^{\bfi \theta}u(-L),\ u'(L)=e^{\bfi \theta}u'(-L).\lb{3.54}
\end{align}
 Finally, we define a $2n$ dimensional subspace of $\C^{4n}$ by
 \begin{equation}
 G_{\lambda,t}:=\{(u(-L), u(L), -u'(-L), u'(L))^\top: u \text{\ is a solution to \eqref{3.53}}\},	\lb{3.55}
 \end{equation}
 and observe that $\ker(H_{\theta}(t)-\lambda)\not=\{0\}$ if and only if $G_{\lambda,t}\cap F^1_{\theta}\not=\{0\}$(cf. \eqref{1.8}). As in Section \ref{SectionTheta}, using \eqref{2.4} we obtain an equivalent to \eqref{3.53},\eqref{3.54} $4n\times 4n$  system
 \begin{align}
 &-y''(x)+t^2(V(tx)\otimes I_2)y(x)=t^2\lambda y(x),\ \lambda\in\R\lb{3.59},\\
 &\qquad y(L)=(I_n\otimes M_{\theta})y(-L),\ y'(L)=(I_n\otimes M_{\theta})y'(-L)\lb{3.57}
 \end{align} 
and denote $\cG_{\lambda,t}:=\{(y(-L), y(L), -y'(-L), y'(L))^\top: u \text{\ is a solution to \eqref{3.59}}\}.$
Applying Proposition \ref{p2.1} with $V(x):=t^2V(tx)$ and $\lambda:=t^2\lambda$, one concludes that the subspace $\cG_{\lambda,t}\subset \R^{8n}$ is in fact Lagrangian with respect to the symplectic form $\omega$. Since $\left(-\partial_x^2\right)_{\theta}\geq 0$ and the potential $V$ is bounded, the spectrum of $H_{\theta}(t)$ is uniformly bounded from below, that is,
\begin{equation}\lb{3.61}
 \Sp (H_{\theta}(t))\subset (\lambda^{\infty},\infty), \text{\ for some\ } \lambda^{\infty} \text{\ and all\ } t\in(0,1].
\end{equation}
Using the parametrization of the boundary of the rectangle in Figure 1(II),
 \begin{align}
 &\lambda(s)=s,\, t(s)=\tau ,\, s\in\Sigma_1:=[\lambda^{\infty},r],\lb{3.62}\\
 &\lambda(s)=r,\, t(s)=s+\tau-r ,\, s\in\Sigma_2:=[r,1+r-\tau ],\lb{3.63}\\
 &\lambda(s)= -s+1-\tau+2r ,\, t(s)= 1,\, s\in\Sigma_3:=[1-\tau+r ,1-\tau -\lambda^{\infty}+2r],\lb{3.64}\\
 &\lambda(s)=\lambda^{\infty},\, t(s)=-s+2-\tau -\lambda^{\infty}+2r,\lb{3.65}\\
 &\hskip3cm s\in\Sigma_4:=[1-\tau -\lambda^{\infty}+2r, 2(1-\tau) -\lambda^{\infty}+2r],\no
 \end{align}
 we introduce the path $s\mapsto \cG_{\lambda(s),t(s)}$ for any $r>\lambda_{\infty}$ and $\tau\in(0,1]$.
 \begin{theorem}\lb{t3.1}
 Assume that $V\in C^1([-L,L], \R^{n\times n}),\ V=V^{\top}$ and fix $\tau\in(0,1]$ and $\theta\in [0,2\pi)$. Then for any $r>\lambda_{\infty}$,
 \begin{equation}\lb{3.66}
 \cN(r,\tau)-\cN(r,1)= {}^{1}/_{2}\mi (\cG_{r,t}|_{\tau\leq t\leq 1},\cF^1_{\theta}).
 \end{equation}
 \end{theorem}
 \begin{proof}
Since the parametrization \eqref{3.62}-\eqref{3.65} defines a loop $s\mapsto \cG_{\lambda(s),t(s)}$, the Maslov index is equal to zero,
$\mi(\cG_{\lambda(s),t(s)},\cF^1_{\theta})=0$,
moreover, by a standard property of the Maslov index,
	\begin{align}
	\mi(\cG_{\lambda(s),t(s)},\cF^1_{\theta})=&\mi(\cG_{\lambda(s),t(s)}|_{\Sigma_1},\cF^1_{\theta})+\mi(\cG_{\lambda(s),t(s)}|_{\Sigma_2},\cF^1_{\theta})\no\\
	&+\mi(\cG_{\lambda(s),t(s)}|_{\Sigma_3},\cF^1_{\theta})+\mi(\cG_{\lambda(s),t(s)}|_{\Sigma_4},\cF^1_{\theta})\lb{3.68}.
	\end{align}
Next we analyze each term in \eqref{3.68}. The calculations in Step 1 of the proof of Theorem \ref{t2.2} for $\Sigma_1$ can be repeated in the current settings i.e. with the potential $\tau^2V(\tau x)$ and the spectral parameter $\tau^2\lambda$. Applying \eqref{2.17} with $\lambda$ replaced by $\lambda\tau^2$, one has the formula for the crossing form at a conjugate point  $s_*\in\Sigma_1$,
\begin{equation}\lb{3.69}
\cQ_{s_*, \cF^1_{\theta}}(\tr(y^0),\tr(y^0))=-\tau^2\|y^0\|^2_{L^2([-L,L])},
\end{equation}
where $y^0$ solves the boundary value problem \eqref{3.59},\eqref{3.57} with $t=\tau$. Therefore, the Maslov index of the part of the loop restricted to $\Sigma_1$ can be computed as previously,
	$\mi(\cG_{\lambda(s),t(s)}|_{\Sigma_1},\cF^1_{\theta})=-2\cN(r,\tau)$;
likewise, taking into account the orientation,
$\mi(\cG_{\lambda(s),t(s)}|_{\Sigma_3},\cF^1_{\theta})=2\cN(r,1)$.
The uniform boundedness of the spectrum of $H_{\theta}(t)$, cf. \eqref{3.61}, rules out existence of conjugate points in $\Sigma_4$, that is 
$\mi(\cG_{\lambda(s),t(s)}|_{\Sigma_4})=0$.
Combining this together, 
we obtain \eqref{3.66}.
 \end{proof}
 An analog of \eqref{3.66} holds for the function $\tilde{\cN}([r_1,r_2),t)=\cN(r_2,t)-\cN(r_1,t)$ counting the numbers of eigenvalues of $H_{\theta}(t)$ in the interval $[r_1,r_2)$. Namely, 
 \begin{align}
 		\tilde{\cN}([r_1,& r_2),t_2)-\tilde{\cN}([r_1,r_2), t_1)\no\\
 		&={}^{1}/_{2} \mi(\cG_{r_2,t_1\leq t\leq t_2},\cF^1_{\theta})-{}^{1}/_{2} \mi(\cG_{r_1,t_1\leq t\leq t_2},\cF^1_{\theta}),\ r_1<r_2.\lb{3.73}
 \end{align}
 
 In conclusion we provide a sufficient condition for crossings to be sign definite in $\Sigma_2$, or, in other words, for the monotonicity of the Maslov index. This result can be viewed as a version of the celebrated Morse-Smale theorem \cite{S65}. We recall that $\mo(H_{\theta})$, the Morse index, is the number of negative eigenvalues of $H_{\theta}$ counting their multiplicities. 
 \begin{theorem}\lb{4.2}
Assume that $V\in C^1([-L,L], \R^{n\times n}), V=V^{\top}$ and fix $\theta\in [0,2\pi)$. 
\begin{itemize}\item[(i)]
If $\Sp(V(x))\subset (-\infty,0]$ a.e. or $\Sp(2tV(tx)-t^2V'(tx))\subset (-\infty,0)$  for all $t\in(0,1]$ then 
\begin{equation}\lb{3.73.2}
\mo(H_{\theta})-\mo(H_{\theta}(\tau))=\sum\nolimits_{\tau\leq t < 1}\dim_{\C}\ker H_{\theta}(t).
\end{equation}
\item[(ii)] If $\Sp(2tV(tx)-t^2V'(tx))\subset (0,\infty)$ then 
\begin{equation}\lb{3.73.1}
\mo(H_{\theta}(\tau))-\mo(H_{\theta})=\sum\nolimits_{\tau< t \leq 1}\dim_{\C}\ker H_{\theta}(t).
\end{equation}
\end{itemize}
\end{theorem}
\begin{proof}
We will use \eqref{3.66} with $r=0$, and compute the Maslov index 
\begin{equation}\lb{3.74}
{}^{1}/_{2}\mi (\cG_{0,t}|_{\tau\leq t\leq 1},\cF^1_{\theta}).
\end{equation}
Let $t_*\in [\tau,1]$ be a conjugate point, i.e. $\cG_{0,t_*}\cap\cF^1_{\theta}\not=\{0\}$. There exists a neighborhood $\Sigma_{t_*}\subset(\tau,1)$ of $t_*$ and a family $(t+t_*)\mapsto R_{(t+t_*)}$ in $C^1(\Sigma_{t_*}, \cL(\cG_{0,t_*}, \cG_{0,t_*}^{\perp}))$, $R_{s_*}=0_{8n}$, such that $\cG_{0,t}=\{u+R_{(t+t_*)}u\big| u\in \cG_{0,t_*}\}$ for all $(t+t_*)\in \Sigma_{t_*}$ (cf.\ discussion prior Definition \ref{def21}). Let us fix a solution $y^0$ to \eqref{3.59},\eqref{3.57} with $\lambda=0$ (this solution exists since $t_*$ is a conjugate point), and consider the family $\tr(y^0_t):=\tr(y^0)+R_{(t+t_*)}\tr(y^0)$ with small $|t|$. We calculate the crossing form using that $y_t^0$ solves \eqref{3.59} with $\lambda=0$ and that $y^0$ satisfies boundary conditions \eqref{3.57},
\begin{align}
&\omega(\tr(y^0), R_{(t+t_*)}\tr(y^0))\no\\
&=(y^0({-L}),- (y_t^0({-L})-y^0({-L}))')_{\R^{2n}}+\left(y^0(L),{y^{0}_t}'(L)-{y^0}'(L)\right)_{\R^{2n}}\no\\
&+(-{y^0}'({-L}),- (y_t^0({-L})-y^0({-L})))_{\R^{2n}}+\left({y^0}'(L), -(y_t^0(L)-y^0(L))\right)_{\R^{2n}}\no\\
&=\int_{-L}^L(-{y^0}'',y_t^0-y^0)_{\R^{2n}}-(y^0,-({y}^0_t-y_0)'')_{\R^{2n}}dx\no\\
&=\int_{-L}^L\Big(y^0(x),\big((t+t_*)^2V((t+t_*)x)\otimes I_2-t_*^2V(t_*x)\otimes I_2\big)y^0_t(x)\Big)_{\R^{2n}}dx.\lb{3.75}
\end{align}
Differentiating with respect to $t$ at $t=0$ yields
\begin{align}
&\cQ_{t_*,\cF_{\theta}^1 }(\tr(y^0),\tr(y^0)):=\frac{d}{dt}\omega(\tr(y^0), R_{(t+t_*)}\tr(y^0))\big|_{t=0}\no\\ 
&=\int_{-L}^L\Big(y^0(x),\big(2t_*V(t_*x)\otimes I_2-t_*^2 V'(t_*x)\otimes I_2\big)y^0(x)\Big)_{\R^{2n}}dx.\lb{3.77}
\end{align}
Using \eqref{2.4},\eqref{3.54.2} and integrating \eqref{3.77} by parts, one infers
\begin{align}
&\int_{-L}^L(y^0(x),\left(2t_*V(t_*x)\otimes I_2-t_*^2V'(t_*x)\otimes I_2\right)y^0(x))_{\R^{2n}}dx\no\\
&=t_*\re (u(-t_*L),V(-t_*L)u(-t_*L))_{\C^n}-t_*^{-1}\sum\limits_{j=1}^{n}|u'_j(-t_*L)|^2.\lb{3.78}
\end{align}
Using the hypothesis, equation \eqref{3.78} implies \eqref{3.73.2},\eqref{3.73.1}.
\end{proof}

\end{document}